\documentclass[review]{elsarticle}
\usepackage[utf8]{inputenc}
\usepackage{amsmath}
\usepackage{amsfonts}
\usepackage{amssymb}
\author{Youssef Rami}
\usepackage{pstricks}

\title{A new invariant that's a lower bound of LS-category}
\date{1 Décembre 2014}
\newtheorem{thm}{Theorem}[subsection]
\newtheorem{cor}[thm]{Corollary}
\newtheorem{lem}[thm]{Lemma}
\newtheorem{prop}[thm]{Proposition}
\newtheorem{defn}[thm]{Definition}
\newtheorem{rem}[thm]{Remark}
\newproof{pf}{Proof}
\numberwithin{equation}{subsection}
\address{D\'epartement de Math\'ematiques \& Informatique, Universit\'e  My Ismail, B. P. 11 201 Zitoune, Mekn\`es, Morocco, Fax: (212) 5 35 53 68 08}

\ead{yousfoumadan@gmail.com, yousseframi22@yahoo.fr}

\begin{document}
\begin{abstract} 

Let $X$ be  a  simply connected CW-complex of finite type and $\mathbb{K}$ any field. 
A first known lower bound of LS-category $cat(X)$ is the Toomer invariant $e_{\mathbb{K}} (X)$ (\cite{Too}). 
In $1980$'s Félix et al. introduced the concept of {\it depth} in algebraic topology and proved  the depth theorem: $depth (H_*(\Omega X, \mathbb{K})) \leq cat(X)$.  

 In this paper, we use  the Eilenberg-Moore spectral sequence of $X$
to introduce a new numerical invariant, denoted by $\textsc{r}(X, \mathbb{K})$, and show that it  has the same properties  as  those of $e_{\mathbb{K}} (X)$.
  When  the evaluation map (\cite{FHT88}) is non-trivial and $char(\mathbb{K})\not = 2$, we prove that $\textsc{r}(X, \mathbb{K})$   interpolates  $depth(H_*(\Omega X, \mathbb{K}))$ and  $e_{\mathbb{K}} (X)$. Hence, we obtain  an improvement of L. Bisiaux theorem (\cite{Bis99}) and then of the depth theorem.
  
   Motivated by these results, we  associate to any commutative differential  graded algebra  $(A,d)$,  a purely algebraic invariant $\textsc{r}(A,d)$
and, via the  theory of minimal models, we relate it with our previous topological results.
 In particular, if $(\Lambda V,d)$  is  a Sullivan minimal  algebra such that  $d=\sum_{i\geq k}d_i$ and $d_i(V)\subseteq \Lambda ^iV$, a greater lower bound is obtained, namely  $e_0(\Lambda V, d)\geq \textsc{r}(\Lambda V, d) + (k-2)$.
\end{abstract}


%%% ----------------------------------------------------------------------
\maketitle
%%% ----------------------------------------------------------------------
%\texttt{\texttt{}}
\section{\bf Introduction}
In 1934, LS-category was introduced by  L. Lusternik and L.
Schnirelman   in connection with variational problems~\cite{LS34}.
They showed that for any closed manifold $M$, this category, denoted
$cat(M)$, is a  lower bound of a number of critical points that any
smooth function on $M$ must have. Later, it was shown that this is
also right for a manifold with a boundary ~\cite{Pal66}.
%\nocite{FHT83}
 If $X$ is a
topological space,  $cat(X)$ is the least integer $n$ such that $X$
is covered  by $n+1$ open subsets $U_i$, each contractible in $X$. This
is an invariant of homotopy type (c.f. ~\cite{Jam78} or ~\cite{FHT01}
Prop 27.2 for example). Though its definition seems easy, the
calculation of this invariant is hard to compute. In ~\cite{Too} G. H. Toomer has introduced a new invariant  $e_{\mathbb{K}}(X)$ ($\mathbb{K}$ being an arbitrary field)  that  bounds  $cat(X)$ by the lower. This one  is  difficult to define, but at first sight, it may look easier to
compute it. Later, in \cite{FHLT89}, F\'elix et al. proved  the  so called {\it the depth theorem }. which state that $depth (H_*(\Omega X, \mathbb{K})) \leq cat(X)$. Recall that 
$depth (H_*(\Omega X, \mathbb{K}))$ is the least integer $p$  such that  $ Ext^{i}_{H_*(\Omega X, \mathbb{K})}(\mathbb{K}, H_*(\Omega X, \mathbb{K})\not = 0$, $i\geq p$ or $\infty$ if $ Ext_{H_*(\Omega X, \mathbb{K})}(\mathbb{K}, H_*(\Omega X, \mathbb{K})\equiv 0$.

Recall also that a version of the Toomer invariant is given 
in terms of the Milnor-Moore spectral sequence 
\begin{eqnarray}\label{1} Ext^{p,q}_{H_*(\Omega X, \mathbb{K})}(\mathbb{K},  \mathbb{K})\Longrightarrow
H^{p+q}(X, \mathbb{K})
\end{eqnarray}
(cf. ~\cite{FH82}, Prop. 9.1) by  
$e_{\mathbb{K}}(X) = sup\{ p\in \mathbb{N} \mid E_{\infty}^{p,q} \not = 0\}$  or $\infty$ if such $p$ doesn't exist.
Ten years after,  L. Bisiaux in \cite{Bis99}, gave an improvement of the depth theorem by showing that for a large class of spaces one have $depth(H_*(\Omega X, \mathbb{K}) )\leq e_{\mathbb{K}}(X).$

Our first gaol in this paper is to give a strong improvement of L. Bisiaux result and hence of the depth theorem. For this,  we introduce a new homotopical  invariant. In fact, inspired  by the last definition of the Toomer invariant, we use the  Eilenberg-Moore spectral sequence 
\begin{eqnarray}\label{2}
 E_2^{p,q} =Ext_{H_*(\Omega X, \mathbb{K})}^{p,q}( \mathbb{K}, H_*(\Omega X, \mathbb{K})) \Rightarrow \mathcal{E}xt_{C_*(\Omega X, \mathbb{K})}^{p+q}(\mathbb{K}, C_*(\Omega X,  \mathbb{K}))
 \end{eqnarray}
to make the following
 \begin{defn}
 Let $X$ be   a simply connected CW-complex of finite type and $\mathbb{K}$ any field. We define 
 $\textsc{r}(X, \mathbb{K}) = sup\{ p\in \mathbb{N} \mid  \mathcal{E}^{p,*} _{\infty}\; \not = 0\; \}$ or $\infty $ if such $p$  doesn't exist.
 When $\mathbb{K}= \mathbb{Q}$, we denote $ \textsc{r}(X, \mathbb{Q})$ by $\textsc{r}_0(X)$.
 \end{defn} 

 It is clear that this is  invariant by homotopy type of $X$. In order to make an algebraic study of it, we  first associate to any  Sullivan minimal algebra $(\Lambda V,d)$ (cf. Section 2)  the following spectral sequences: 
\begin{eqnarray}\label{3} H^{p,q}(\Lambda V, d_k)\Longrightarrow
H^{p+q}(\Lambda V, d)
\end{eqnarray}
\begin{eqnarray}\label{4} \mathcal{E}xt_{(\Lambda V, d_k)}^{p,q}(\mathbb K, (\Lambda V,
d_k))\Longrightarrow \mathcal{E}xt_{(\Lambda V, d)}^{p+q}(\mathbb K,
(\Lambda V, d)).\end{eqnarray}
Here,  $\mathcal{E}xt$ is  the "differential"-Ext of Eilenberg-Moore \cite{Moo59} and $k$ is such that $d=\sum _{i\geq k}d_i$.  We call these, {\it  generalized (resp. $\mathcal{E}xt$-version generalized) Milnor-Moore spectral sequence}.

Now, given   a $1$-connected commutative  differential graded algebra (cdga for short) $(A,d)$ over  a field $\mathbb{K}$ of $char(\mathbb{K})\not = 2$. Denote by $(\Lambda V,d)$ its Sullivan minimal model (cf. section 2) and  by $\mathcal{E}^{*,*}_{\infty}$ the $\infty$ term of the generalized $\mathcal{E}xt$-version Milnor-Moore spectral sequence \ref{4}. We associate to  $(A,d)$  the following integer:
$$\textsc{r}(A,d) := \textsc{r}(\Lambda V ,d) := sup\{ p\in \mathbb{N}\;  | \:  \mathcal{E}^{p,*} _{\infty}\; \not = 0\; \}$$
with the convention  $\textsc{r}(A,d)  = \infty $ if such $p$ doesn't exists.
In fact,  by Remark $1.3. (1)$ in \cite{FHT88}, this is clearly an (algebraic) invariant which has   properties similar to those  of Toomer invariant (cf.  Remark 4.0.3 for more details). For that raison, we call it the {\it Ext-version  Toomer invariant} of $(A,d)$.

 Recall that to any  nilpotent space $X$ (in particular to
any simply connected space), D. Sullivan  in ~\cite{Sul78}
 associated a unique (up to quasi-isomorphism) minimal cdga $(\Lambda V, d)$ over the rational field $\mathbb{Q}$. This is called the Sullivan minimal model of $X$.
Our first main result reads: 
\begin{thm}
{\it Let $X$ be  a simply connected CW-complex of finite type and $(\Lambda V, d)$
its  Sullivan minimal model. Then, the  $\mathcal{E}xt$-version generalized-Milnor-Moore spectral sequence of  $(\Lambda V,d)$  is isomorphic to the Eilenberg-Moore spectral sequence. }
\end{thm}
%\begin{rem}

In \cite{Bis99}, L. Bisiaux introduced the following spectral sequence:
\begin{eqnarray}\label{5} \mathcal{E}xt_{(T(W), d_2)}^{p,q}(\mathbb K, (T(W) , d_2)
)\Longrightarrow \mathcal{E}xt_{(T(W), d)}^{p+q}(\mathbb Q,
(T(W), d)),
\end{eqnarray}
where $(T(W) ,d)$ designate a free model of $C_*(\Omega X, \mathbb{K})$   over an arbitrary field $\mathbb{K}$ (\cite{HL88}).

 Using the first step (which is valid for any field $\mathbb{K}$) in the proof of the previous theorem, we see that  the two spectral sequences \ref{2} and \ref{5} are isomorphic from there second terms. Accordingly, the same arguments as that applied in the proof of L. Bisiaux  main result in \cite{Bis99} gives our improvement for this one as follow: 
 \begin{thm} Let $X$ be a simply connected CW-complex such that each $H_i(X, \mathbb{K})$ is finite dimensional. If  $ev_{C^*(X, \mathbb{K})}$ is non-zero, then $$depth H_*(\Omega X, \mathbb{K})\leq \textsc{r}(X,\mathbb{K})\leq e_{\mathbb{K}}(X).$$
\end{thm}
(The definition of $ev_{C^*(X, \mathbb{K})}$ is recalled in \S 2.2).

For other consequences of  theorem 1.0.2,  note that in one hand, in the rational case,  we have immediately the 
\begin{cor} Let $X$ be as in  the  theorem 1.0.2 with Sullivan minimal model $(\Lambda V,d)$.  Then,
%\begin{enumerate}
$\textsc{r}_0(X)= \textsc{r}(\Lambda V, d$).
%\item If in addition $X$ is  elliptic, then $R_0(X)$ is a lower bound of $cat_0(X)$.
%\end{enumerate} \stackrel{\simeq}{\longrightarrow} C^*(L)
\end{cor}

On the other hand, suppose that  $char(\mathbb{K}) = p > 2$ and let $X$ be a space in the rang of  Anick (\cite{Ani89}), that is $X$ is a finite $r$-connected CW-complex with $r\geq 1$ and $dim(X)\leq rp $. 
In \cite{Hal92}, S. Halperin associated to a such space an unique (up to quasi-isomorphism) Sullivan minimal model $(\Lambda W,d)$. Using properties of this latter (cf. \S 2.1 for more details) and  similar arguments as in the proof of theorem 1.0.2, we  obtain the following
\begin{cor}
Let  $\mathbb{K}$  any field with $char(\mathbb{K}) = p > 2$, $X$  a space in the rang of  Anick and  $(\Lambda V,d)$ its Sullivan minimal model, then $\textsc{r}(X, \mathbb{K})=\textsc{r}(\Lambda V,d)$.
\end{cor}
%The proof of this theorem is a part of that of our main result which we will give soon. For this purpose, 

Returning to the context of rational homotopy theory, in ~\cite{FH82}, Y.
F\'elix and S. Halperin developed a deep approach for computing the
LS-category $cat(X_{\mathbb{Q}})$ of the rationalisation of a $1$-connected space $X$. This is done  by the use of its Sullivan minimal model. Later on, Y. F\'elix, S.
Halperin and J.M. Lemaire ~\cite{FHL} showed that for Poincar\'e
duality spaces, $cat(X_{\mathbb{Q}})$ coincides with the rational Toomer invariant
 $e_{\mathbb{Q}}(X)$.

 Recall for the remainder that a space $X$ is said rationally elliptic, if it has an elliptic  Sullivan minimal  model $(\Lambda V,d)$  such that  $dimV < \infty $ and $dim(H(\Lambda V,d)<\infty $.  Denote
$V^{even}=\bigoplus _{k}V^{2k}$,  $V^{odd}=\bigoplus
_{k}V^{2k+1}$ and  put $d=\sum _{i\geq k}d_i$ with $d_i(V)\subseteq \Lambda ^iV$. Minimality of $(\Lambda V,d)$ requires that $k\geq 2$. 
The main result by L. Lechuga and A. Murillo in
~\cite{LM02}  states  that  if   $(\Lambda V,d_k)$ is also elliptic, then $e_0(\Lambda V, d) = dim (V^{odd})+ (k-2)dim (V^{even})$.

 Our second  goal,  in this paper, is to approach  the  LS-category of the large class of rationally  elliptic  simply connected CW-complex $X$ for which $(\Lambda V,d_k)$  is not necessarily elliptic. 
 % For this, we introduce a new (topological) invariant. In fact, inspired  by the definition of the Toomer invariant  through the Milnor-Moore spectral sequence \ref{SS(2)}, we use the  Eilenberg-Moore spectral sequence :
 %$$E_2^{p,q} =Ext_{H_*(\Omega X, \mathbb{K})}^{p,q}( \mathbb{K}, H_*(\Omega X, \mathbb{K})) \Rightarrow \mathcal{E}xt_{C_*(\Omega X, \mathbb{K})}^{p+q}(\mathbb{K}, C_*(\Omega X,  \mathbb{K}))$$
%to make the following:
%Recall also that a graded differential algebra $(A,d)$ is a Gorenstien algebra if $dim( \mathcal{E}xt_{(A, d)}^{*,*}(\mathbb K, (A,
%d)))=1$ (\cite{FHT88}) and that $(\Lambda V, d)$ is a Gorenstein algebra if and only if $dim(V)<\infty$ (\cite{Mur94}). It follow immediately that there is a unique $(p,q)$ such that $ \mathcal{E}xt_{(\Lambda V, d)}^{p,q}(\mathbb Q, (\Lambda V,d))\; \not = 0$ and it is one dimensional. 
 In this direction, our first  result gives a more precise formula  of our improvement of L. Bisiaux theorem as follow (for the case where $char(\mathbb{K}) = p >2$, see the final remark below) :
\begin{thm} {\it If $X$ is   a  rationally elliptic finite type simply connected  CW-complex  and $(\Lambda V, d)$
its  Sullivan minimal model, then $cat_0 (X)  \geq \textsc{r}_0(X) + (k-2)$.}
\end{thm}  

To finish this introduction, recall that  to any cdga $(\Lambda V, d)$, S. Halperin  associated in \cite{Hal77},  another one, said {\it pure} and denoted  $(\Lambda V,d_{\sigma})$, with the property:  $(\Lambda V, d)$ is elliptic if and only if $(\Lambda V,d_{\sigma})$ is. 
By the use of \ref{4}, we have:
\begin{prop}  With the notations above, if $(\Lambda V,d)$ is a  Sullivan minimal algebra such that $dim(V)<\infty$, then $\textsc{r}(\Lambda V,d) = \textsc{r}(\Lambda V,d_k)$. As a particular case, we have $
\textsc{r}(\Lambda V,d_{\sigma }) = \textsc{r}(\Lambda V,d_{\sigma , k})$.
\end{prop}

Our second result in this context gives an explicit formula  for $\textsc{r}(\Lambda V,d_{\sigma })$ as follow:

\begin{thm}
Let $(\Lambda V,d)$ any pure   Sullivan minimal algebra.\\ If $dim(V)<\infty $ then $\textsc{r}(\Lambda V,d) = dim(V^{odd})+(k-2)(dim(V^{even})-1).$

\end{thm}

%%% ----------------------------------------------------------------------

\section{\bf Preliminary} In this section we recall some notions we
will use in the other sections. A commutative differential graded algebra (resp. differential graded algebra, differential graded Lie algebra) will be abbreviated by cdga (resp. dga, dgl). The suspension (resp. desuspension,  dual) of an object N will be denoted by $sN$ (resp. $s^{-1}N$,  $N^{\vee}$).

\subsection{A Sullivan minimal model:}

Let $\mathbb K$ be a field of any characteristic $\not = 2$

 A {\it Sullivan algebra} is a free 
cdga  $(\Lambda V, d)$, where $\Lambda V=
Ext(V^{odd}) \otimes Sym(V^{even})$. It is  generated by the
graded $\mathbb K$-vector space  $V= \oplus _{i=0}^{i=\infty}V^i$ which 
 has a well ordered basis $\{x_{\alpha }\}$ such that
$dx_{\alpha }\in \Lambda V_{<\alpha }$ ($V_{< \alpha} = span\{ v_{\gamma},\; \gamma < \alpha \}$). Such algebra is said {\it
minimal } if it has the property :  $deg(x_{\alpha })< deg(x_{\beta })$ implies $\alpha
<\beta $. 
If $V^0= \mathbb{K}$ and  $V^1=0$, this is equivalent to say that $
d(V)\subseteq \oplus _{i=2}^{i=\infty} \Lambda ^iV$.

 A {\it Sullivan model} for a cdga $(A,d)$
 is a quasi-isomorphism (morphism inducing an isomorphism in cohomology) $(\Lambda
V, d)\stackrel{\simeq } \rightarrow (A,d)$ with the  source a Sullivan
algebra.
 If $H^0(A)=\mathbb K$, $H^1(A)=0$ and
$dim(H^i(A,d))<\infty$ for all $i\geq 0$ then (\cite{Hal92}, Th.
7.1) this minimal model exists.
  %If $X$ is a topological space any
 %(minimal) model of the algebra $C^*(X, \mathbb K)$ is said a {\it
 %Sullivan  (minimal) model } of $X$.

The uniqueness (up to quasi-isomorphisms) in the rational case was
assured by D. Sullivan in (\cite{Sul78}). Indeed he associated to any
simply connected space $X$ the {\it cdga} $A(X)$ of polynomial
differential forms and showed that its minimal model $(\Lambda V,
d)$ satisfies
$V^i\cong Hom_{\mathbb Z}(\pi _i(X), \mathbb Q) ;\;\;\; \forall
i\geq 2,$ that is, in the case where $X$ is a finite type
CW-complex, the generators of $V$ correspond to those of $\pi
_*(X)\otimes \mathbb Q$.

Now, assume that   $char(\mathbb K)=p> 2$ and let $X$ be an  $r$-connected CW-complex 
with $dim(X)<rp$ ($r\geq 1$) ie $X$ is in the rang of  Anick (\cite{Ani89}). In  \cite{Hal92}, S. Halperin proved that $(C_*(\Omega (X), \mathbb{K})$ is quasi-isomorphic to the enveloping algebra $UL$ of an appropriate dgl $L$. Using its bar-construction $B(UL)$, he  constructed   a filtration preserving quasi-isomorphism   $\Omega (UL) \stackrel{\simeq } \rightarrow C^*(L)$ from $\Omega (UL) = (B(UL))^{\vee}$ to the Cartan-Chevalley-Eilenberg complex associated  to $L$. In addition, under the hypothesis on $X$  and by theorem $7.1$ in \cite{Hal92}, $C^*(L) = (\Lambda V ,d)$, where $V = (sL)^{\vee}$ (Cor. $6.1$ \cite{Hal92}) has a minimal model  $(\Lambda W, d)$. 

 As I know, there is  no
relationship between $W$ and $\pi _*(X)\otimes \mathbb K$.
  
  Notice also that  in his study of ellipticity of rational spaces (\cite{Hal77}), S. Halperin introduced  the following  spectral sequence: 
  \begin{eqnarray}\label{6}
E_2^{p,q} = H^{p,q}(\Lambda V, d_{\sigma })\Longrightarrow H^{p+q}(\Lambda V,
d)
\end{eqnarray}
whose $E_1$-term $(\Lambda V, d_{\sigma })$   (called the  {\it pure cdga} associated to $(\Lambda V, d)$) is provided with the differential
 $d_{\sigma}$ defined as follow:
  $$d_{\sigma}(V^{even})=0\;\;\;\;\; \hbox{and} \;\;\;\;\;\;
  (d-d_{\sigma})(V^{odd})\subseteq \Lambda
  V^{even}\otimes \Lambda ^+V^{odd}.$$
  
This one (called the {\it odd spectral sequence}) is induced by the filtration \\ $F^{p}= (\Lambda V)^{\geq p, *}$, where  $(\Lambda V)^{n+q,- q} =(\Lambda V^{even} \otimes \Lambda ^q V^{odd})^n$.    

For a completeness    and a subsequent use, we give bellow  the filtration inducing the  $\mathcal{E}xt-version$ {\it odd spectral sequence}  introduced  in \cite{Ram99}:
\begin{eqnarray}\label{7}\mathcal{E}xt_{(\Lambda V, d_{\sigma })}^{p,q}(\mathbb Q, (\Lambda V,
d_{\sigma }))\Longrightarrow \mathcal{E}xt_{(\Lambda V,
d)}^{p+q}(\mathbb Q, (\Lambda V, d)).\end{eqnarray}

Consider first  $(\Lambda V\otimes \Gamma sV, D)$ an acyclic closure of $(\Lambda V,d)$. Then  the differential  $D_{\sigma}$ that agree with $d_{\sigma}$ on $V$ and s.t. $D_{\sigma}(sv)=v-s(d_{\sigma}v), \forall v\in V$ shows that 
$(\Lambda V\otimes \Gamma sV, D_{\sigma})$ is one of $(\Lambda V, d_{\sigma})$.
 
Now let 
$$(\Lambda V\otimes \Gamma sV)^{r,s}= \bigoplus _{
a_1,a_2,a_3,\\
 r=a_1+a_2+2a_3}(\Lambda Q\otimes \Gamma ^{a_1} sQ\otimes \Lambda ^{a_2}P\otimes \Gamma ^{a_3} sP)^s.$$
The filtration on  $(A, \mathcal{D}):=(Hom_{\Lambda V} (\Lambda V\otimes \Gamma sV, \Lambda V), \mathcal{D})$   is then given  by $$\mathcal{F'}^p(A^n)=\bigoplus _{r,s}Hom_{\Lambda V} [(\Lambda V\otimes \Gamma sV)^{r,s}, (F^{p+n+r+s}(\Lambda V))^{n+s}].$$ 
(Recall that $ \mathcal{D}(f)=d\circ f+ (-1)^{|f|+1}f\circ D, \forall f\in A$) 
\begin{rem} 
Let $\varphi \in \mathcal{F'}^p(A^n)$, it is clear that $\varphi (1)\in F^p = \Lambda ^{\geq p}V$.
\end{rem}

%%% ----------------------------------------------------------------------

\subsection{The evaluation map:} Let $(A,d)$ be an
augmented  $\mathbb K$-dga and choose an $(A,d)$-semifree
resolution
 (\cite{FHT88}) $\rho : (P,d) \stackrel{\simeq}\rightarrow (\mathbb
 K,0)$ of $\mathbb K$.
 Providing $\mathbb K$ with the $(A,d)$-module structure induced by
 the augmentation, we define a chain
 map:\\
 $ Hom_{(A,d)}((P,d), (A,d)) \longrightarrow (A,d)$ by
 $f\mapsto f(z)$, where $z\in P$ is a cycle representing $1_{\mathbb{K}}$. Passing to homology, we  obtain the {\it evaluation
 map} of $(A,d)$:
 $$ ev_{(A,d)}:  \mathcal{E}xt_{(A, d)}(\mathbb K,
(A, d)) \longrightarrow  H(A, d),$$
where $\mathcal{E}xt$ is the differential $Ext$ of Eilenberg
and Moore (\cite{Moo59}). Note that the definition of $ ev_{(A,d)}$  is independent on the choice of $(P,d)$ and
$z$. Moreover,   it is natural with respect to $(A,d)$. 

 The authors of
~\cite{FHT88} also defined the concept of  a Gorenstein space over any
field $\mathbb K$. It is a space $X$ such that
$dim\mathcal{E}xt_{C^*(X,\mathbb K)}(\mathbb K, C^*(X,\mathbb
K))=1$. In addition, if $dimH^*(X,\mathbb K)<\infty$,
then $X$ satisfies Poincar\'e duality property over $\mathbb K$ and its
fundamental class is closely related to the evaluation map (\cite{Mur94}).

%%% ----------------------------------------------------------------------

\subsection{The Toomer invariant:} The Toomer invariant is defined by more  than one way.
Here we recall its definition in terms of Sullivan algebras. Let then
$(\Lambda V, d)$ any minimal one over  a field $\mathbb K$. Put
$p_n: \;\;\;\;\; \Lambda V \rightarrow \frac{\Lambda V}{\Lambda ^{\geq
n+1}V}$,  the projection onto the quotient dga obtained by factoring out by the differential graded ideal
generated by monomials of length at least $n+1$. {\it The
"commutative" Toomer invariant } $e_{ \mathbb K}(\Lambda V,d)$ of
$(\Lambda V, d)$ is the smallest integer $n$ (possibly   $\infty $) such that $p_n$ induces
an injection in cohomology.

 In ~\cite{HL88},
S. Halperin and J.M. Lemaire extended  the last definition to free models. Hence they associated to any simply connected finite
type CW-complex $X$, the invariant
$e_{\mathbb{K}}(T(W),d)$ in terms of its free minimal model $(T(W),d)$ over $\mathbb{K}$. They also  showed
that it coincides  with the classical Toomer invariant $e_{\mathbb{K}}(X)$. Consequently
if $char(\mathbb K)\not = 2$ and $X$  is in the rang of  Anick, then, using both its Sullivan minimal
model  and its free minimal model, one  deduce
 that $e_{\mathbb K}(X) = e_{\mathbb
K}(\Lambda V,d) = e_{\mathbb{K}}(T(W),d)$ (cf
 \cite{HL88}, Th 3.3 for more details).

For $\mathbb{K}=\mathbb{Q}$, we shall denote $e_0(\Lambda V,d)$ instead of $e_{\mathbb
Q}(\Lambda V,d)$.

%%% ----------------------------------------------------------------------

\begin{rem} 
\begin{enumerate}
\item The definition given above for  $e_{0}(\Lambda V,d)$ is also expressed in terms of the Milnor-Moore spectral sequence (which coincides with \ref{3},  for $k=2$) (cf. \cite{FH82}) by
 $e_{0}(\Lambda V, d)= sup\{p \mid E_{\infty}^{p,q}\not = 0\}\; or\; \infty.$
\item
 In (~\cite{FH82}, Lemma 10.1), Y. F\'elix and S. Halperin showed  also that
whenever $H(\Lambda V,d)$ has Poincar\'e duality and $\omega$ represents its fundamental class, then $e_{0}( \Lambda V,d)=sup\{k \; /\; \omega \; \hbox{ can be
represented by } \hbox{a cycle in}\;   \Lambda ^{\geq k}V\}.$
   
By Lemma 2.1 in \cite{Bis99},   this characterisation remains true for any field $\mathbb  K$ when we replace $(\Lambda V,d)$ by $(T(W),d)$.
\end{enumerate}
\end{rem}
%\medskip
%%% ----------------------------------------------------------------------
\subsection{The Eilenberg-Moore spectrale sequence:}
Now let $\mathbb{K}$ be any field, $X$    a simply connected CW-complex of finite type and  $\Omega X$ the space of based loops on $X$. Denote by $(A,d)=(C_*(\Omega X),d)$ the chain complex on  $\Omega X$  and $\bar A= ker(\varepsilon : A \rightarrow \mathbb{K})$. So $B(A) = \bigoplus _{n\geq 0}T^n(s\bar A)$ (resp. $A\otimes B(A)= \bigoplus _{n\geq 0}A\otimes T^n(s\bar A)$) is the reduced bar construction on $(A,d)$ with coefficients in $\mathbb{K}$ (resp. in $(A,d)$) (\cite{FHT01}, p. 268). 

 The  filtration   $I\!\!F^q=\bigoplus _{k\leq q}A\otimes T^k(s\bar A)$ on $A\otimes B(A)$ is  such that $d(T^q(s\bar A))\subseteq \bigoplus _{k< q}A\otimes T^k(s\bar A)$.  It is an $A$-semifree resolution of $\mathbb{K}$ and   gives rise to a spectral sequence with first term $E_1^q = H(A,d)\otimes T^q(s\bar A)$. By (\cite{FHT01} pro. 20.11) one can suppose this an $H(A)$-semifree resolution of $\mathbb{K}$.  Hence the filtration on\\  $Hom_A(A\otimes B(A), A)$ defined by: 
 \begin{eqnarray} \label{10}
 \mathcal{I\!\!F}^q=\{ f\in Hom_A(A\otimes B(A), A) \mid f(I\!\!F^k) = 0, \forall k<q \}
 \end{eqnarray}
  induces  the converging Eilenberg-Moore spectral sequence:  
 $$E_2^{p,q} =Ext_{H_*(\Omega X, \mathbb{K})}^{p,q}( \mathbb{K}, H_*(\Omega X, \mathbb{K})) \Rightarrow \mathcal{E}xt_{C_*(\Omega X, \mathbb{K})}^{p+q}(\mathbb{K}, C_*(\Omega X,  \mathbb{K}))$$

%%% ----------------------------------------------------------------------
 
\section{\bf  Generalized Milnor-Moore spectral sequences :} In this section, we will work
over an arbitrary   field $\mathbb{K}$. Let $(\Lambda V,d)$ a {\it cdga }
with $dim(V)<\infty$. Suppose that $d=\sum _{i\geq k}d_i$ and $k\geq
2$. The filtrations that induce the spectral sequences \ref{3}
and \ref{4} given in the introduction are defined respectively as
follow: \begin{eqnarray} \label{8} F^p= \Lambda ^{\geq
p}V=\bigoplus_{i=p}^{\infty }\Lambda ^{i}V,\end{eqnarray}
\begin{eqnarray} \label{9} \mathcal{F}^p=\{f\in Hom_{\Lambda V}(\Lambda V\otimes \Gamma
(sV),\Lambda V)\; \mid \; f(\Gamma (sV))\subseteq \Lambda ^{\geq p}V.
\}\end{eqnarray}

Recall that $\Gamma (sV)$ is the divided power algebra of $sV$ and the differential $D$ on $\Gamma (sV)\otimes\Lambda V $ is a $\Gamma$-derivation
(ie $D(\gamma^p(sv))=D(sv)\gamma ^{p-1}(sv),  p\geq 1, sv\in (sV)^{even}, and D(sv)=v+s(dv)$)
which restrict to $d$ in $\Lambda V$. With this
differential, $(\Lambda V\otimes \Gamma (sV), D)$ is a dga called an
{\it acyclic closure of} $(\Lambda V,d)$, hence it is a $(\Lambda
V,d)$-semifree module. Therefore the projection $(\Lambda V\otimes \Gamma
(sV), D) \stackrel {\simeq } {\longrightarrow} \mathbb K$ is a
semifree resolution of $\mathbb K$.

Recall also that de differential $\mathcal{D}$ of $Hom_{(\Lambda
V,d)}((\Lambda V\otimes \Gamma (sV), D),(\Lambda V,d))$ is defined
by: $ \mathcal{D}(f)=d\circ f+ (-1)^{|f|+1}f\circ D $.

Let us denote $A=\Lambda V$ (resp. $ Hom_{\Lambda V}(\Lambda
V\otimes \Gamma (sV),\Lambda V)$), $G^p =F^p$ (resp.
$G^p=\mathcal{F}^p$) and $\delta = d$ (resp. $\delta = \mathcal{D}$) the differential of $A$.

The  filtrations \ref{8} and \ref{9} verify 
the following lemma and then they define the tow spectral
sequences \ref{3} and \ref{4}.

\begin{lem}
(i) $(G^p )_{p\geq 0}$  is decreasing.\\
(ii) $G^0 (A) =A.$\\
(iii) $\delta (G^p(A)) \subseteq G^p(A) .$
 \end{lem}
 
\begin{pf}
(i) and  (ii) are immediate. The propriety  (iii) follows, first from  the
definition of $\mathcal{D}$ on $Hom_{(\Lambda V,d)}((\Lambda
V\otimes \Gamma (sV), D),(\Lambda V,d))$, second as  $d$ is
minimal and finally because $D(\gamma^p(sv))= D(sv)\gamma ^{p-1}(sv)=
(v+s(dv))\gamma ^{p-1}(sv)$.
\end{pf}

\subsection{Determination of the first terms of the tow spectral
sequences:} The two filtrations \ref{8} and \ref{9} are bounded, so they induce convergent
spectral sequences. We calculate here  there
first terms.

Beginning with the filtration \ref{8}, one can check easily the following:
$$E_0^p=F^p/F^{p+1}\cong  E_1^p\cong \ldots \cong E_{k-2}^p \cong \Lambda ^pV.$$

The first non-zero differential
is $\partial _{k-1}: E_{k-1}^p\rightarrow E_{k-1}^{p+(k-1)}$ which
coincides with the first non zero  term  $d_k$ in the differential $d$
of $(\Lambda V,d)$, and subsequently $(E_{k-1},\partial _{k-1})=(\Lambda V,d_k)$.
So the first term in the induced spectral sequence \ref{3}
is
$E_{k}=H(\Lambda V,d_k).$

For the second spectral sequence \ref{4}, its general term is:
 $$\mathcal {E}_r^p=\frac{ \{ f\in \mathcal{F}^p,\; \mathcal{D}(f)\in \mathcal {F}^{p+r}\}}
  { \{ f\in \mathcal {F}^{p+1},\; \mathcal {D}(f)\in \mathcal {F}^{p+r}\} + \mathcal{F}^p\cap \mathcal{D}(\mathcal{F}^{p-r+1})\} }.$$
  
We first prove the
following important lemma:
\begin{lem}
Let $f\in \mathcal{F}^p$, then for any $p\geq 0$,\\
1- $\mathcal{D}(f)\in \mathcal{F}^{p+2} \Leftrightarrow f\in Ker(
\mathcal{D}_2).$\\
2- $f-\mathcal{D}(g)\in
\mathcal{F}^{p+1} \Leftrightarrow  f-\mathcal{D}_2(g)\in \mathcal{F}^{p+1}
$
\end{lem}
\begin{pf} Using the relations
 $ \mathcal{D}(f)=d\circ f+ (-1)^{|f|+1}f\circ D $, $D(sv)=v+s(dv)$ and $D(\gamma^p(sv))= D(sv)\gamma ^{p-1}(sv)$, we have successively:
\\
1- $\mathcal{D}(f)\in
\mathcal{F}^{p+2}\Leftrightarrow
\mathcal{D}(f)(\Gamma (sV))\cap \Lambda ^{p+1}V=\{0\} \Leftrightarrow
\mathcal{D}_2(f)(\Gamma (sV))\cap \Lambda ^{p+1}V =\{0\} \Leftrightarrow f\in Ker(\mathcal{D}_2)$
and
\\
2- $f-\mathcal{D}(g)\in
\mathcal{F}^{p+1} \Leftrightarrow
 (f-\mathcal{D}(g))(\Gamma (sV))
\cap \Lambda ^{p}V =\{0\}\Leftrightarrow (f-\mathcal{D}_2(g))\cap \Lambda ^{p}V =\{0\}$.
\end{pf}

 Whence
$(\mathcal {E}_0^p, d_0)=(\mathcal {F}^p/\mathcal {F}^{p+1}, 0)$ and 
$(\mathcal {E}_1^p, d_1)=(\mathcal {E}_0^p, d_1)=(\mathcal{ F}^p/\mathcal {F}^{p+1}, d_1)$,
where $d_1: \mathcal {F}^p/\mathcal {F}^{p+1}\rightarrow
\mathcal{F}^{p+1}/\mathcal{F}^{p+2}$ is such that $\forall f\in
\mathcal {F}^p$, $\forall  g\in \mathcal {F}^{p-1}$:
$$
 \left \{ \begin{array}{l}
   d_1(\bar f)=\bar 0\Leftrightarrow \mathcal{D}(f)\in
\mathcal{F}^{p+2}\Leftrightarrow
f\in Ker(
\mathcal{D}_2)\Leftrightarrow \overline{\mathcal{D}}_2(\bar f) =\bar 0\\
 \bar f- d_1(\bar g)=\bar 0\Leftrightarrow f-\mathcal{D}(g)\in
\mathcal{F}^{p+1} \Leftrightarrow f-\mathcal{D}_2(g)\in
\mathcal{F}^{p+1} \Leftrightarrow \bar f = \overline{\mathcal{D}}_2(\bar g)
\end{array}\right.
$$

As a consequence:

If $k=2$,
$\mathcal {E}_2^*= \mathcal
{E}xt_{(\Lambda V, d_2)}^*(\mathbb{Q}, (\Lambda V, d_2)),$ which is
exactly the first term of the Ext-Milnor-Moore spectral sequence
introduced by A. Murillo in  ~\cite{Mur94}.

If $d_2 = 0$, the differential ${\mathcal{D}}_2$ is reduced to ${\mathcal{D}}_0$, where ${\mathcal{D}}_0(f)= (-1)^{|f|+1}f\circ D_0 $, $D_0(\gamma^p(sv))= D_0(sv)\gamma ^{p-1}(sv)=
v\gamma ^{p-1}(sv)$. Then
$\mathcal {E}_2^*= \mathcal
{E}xt_{(\Lambda V, 0)}^*(\mathbb{Q}, (\Lambda V, 0)).$

Consider  now  the case were $k = 3$ that is $d_3 \not = 0$:

The general term is then:
$$\mathcal {E}_3^p=\frac{ \{ f\in
\mathcal{F}^p,\; \mathcal{D}(f)\in \mathcal {F}^{p+3}\}}
  { \{ f\in \mathcal {F}^{p+1},\; \mathcal {D}(f)\in \mathcal {F}^{p+3}\} + \mathcal{F}^p\cap \mathcal{D}(\mathcal{F}^{p-2}) }$$

  Notice that: $$f\in \mathcal {F}^{p} \; \hbox{and} \; \mathcal {D}(f)\in \mathcal {F}^{p+3}\Leftrightarrow f\in \mathcal {F}^p\cap Ker(\mathcal {D}_3),$$
 therefore   the natural projection $\gamma : \mathcal {F}^p\cap Ker(\mathcal {D}_3)\rightarrow \mathcal {E}_3^p$ is well defined and we have:
 \begin{lem} The morphism ${\gamma}:  \mathcal {F}^p\cap Ker(\mathcal {D}_3) \rightarrow \mathcal {E}_3^p$ of graded $\mathbb{Q}$-vector spaces 
 is   surjective  and its Kernel is:  $Ker(\gamma) = \mathcal {F}^p \cap Im(\mathcal {D}_3)+\mathcal {F}^{p+1}\cap Ker(\mathcal {D}_3)$.
 \end{lem}
 \begin{pf} $\gamma$ is clearly a surjective morphism of graded $\mathbb{Q}$ vector-spaces. Let $f\in Ker (\gamma )= \{ \varphi \in \mathcal {F}^{p+1},\; \mathcal {D}(\varphi )\in \mathcal {F}^{p+3}\} + \mathcal{F}^p\cap \mathcal{D}(\mathcal{F}^{p-2}) $. Write $f=g+\mathcal {D}(h)= \mathcal {D}_3 (h)+(g+\mathcal {D}_{\geq 4}(h))$, where $g\in \mathcal {F}^{p+1}, \mathcal {D}(g)\in \mathcal {F}^{p+3}$ and $h\in \mathcal{F}^{p-2}$. As $\mathcal {D}_3(h)\in \mathcal {F}^p$ and  $\mathcal{D}_3(f)=0$, we have $g+\mathcal {D}_{\geq 4}(h)\in \mathcal {F}^{p+1}\cap Ker(\mathcal {D}_3)$.
\end{pf}

 Consequently, we have the following isomorphisms of graded $\mathbb{Q}$-vector spaces:
$$ \mathcal {E}_3^p \cong \frac{\mathcal {F}^p\cap Ker(\mathcal {D}_3)}{\mathcal {F}^p\cap Im(\mathcal {D}_3)+\mathcal {F}^{p+1}\cap Ker(\mathcal {D}_3)} \cong \frac{\frac{\mathcal {F}^p\cap Ker(\mathcal {D}_3)}{\mathcal {F}^p\cap Im(\mathcal {D}_3)}}{\frac{\mathcal {F}^p\cap Im(\mathcal {D}_3)+\mathcal {F}^{p+1}\cap Ker(\mathcal {D}_3)}{\mathcal {F}^p\cap Im(\mathcal {D}_3)}} \cong \frac{\frac{\mathcal {F}^p\cap Ker(\mathcal {D}_3)}{\mathcal {F}^p\cap Im(\mathcal {D}_3)}}{\frac{\mathcal {F}^{p+1}\cap Ker(\mathcal {D}_3)}{\mathcal {F}^{p+1}\cap Im(\mathcal {D}_3)}}.$$
That is:
$$\mathcal {E}_3^p \cong \frac{H(\mathcal {F}^p, {\mathcal { D}_3}_{\mid \mathcal {F}^{p}})}{H(\mathcal {F}^{p+1}, {\mathcal { D}_3}_{\mid \mathcal {F}^{p+1}})}.$$
Hence, 
$\mathcal {E}^*_3\cong \oplus _{p\geq 0}\mathcal {E}xt^p_{(\Lambda V,d_3)}(\mathbb{Q} , (\Lambda V,d_3))$ is an isomorphism of graded algebras.

 The same arguments  used for $k=3$ can be applied term by term to conclude that for any $k\geq 3$, the first term of the second spectral sequence \ref{4} is:
$$\mathcal {E}^*_k\cong \mathcal {E}xt^*_{(\Lambda V,d_k)}(\mathbb{Q} , (\Lambda V,d_k)).$$

\section{\bf Proof of the main results}
In this paragraph we give  proofs of our main results  (Theorem $1.0.2$, Theorem $1.0.6$ and Theorem $1.0.8$).

\centerline{\bf Proof of theorem $1.0.2$}
\begin{pf}
We proceed in three steps:
\begin{enumerate}[Step 1.]
  \item 
   As in the preliminary, denote $A=C_*(\Omega X, d)$,  $B(A)$ its reduced bar construction and   $\Omega (A) = T(s^{-1}\overline{A^{\vee}})$ the dual of $B(A)$. 
This part of the proof reposes on    the  following isomorphisms (cf. proof of Theorem 2.1 in \cite{FHT88}):
$$\begin{array}{ccc}
(Hom_A(A\otimes B(A), A), d) & \stackrel{\cong \varphi _{A}}{\longrightarrow} & (End_{A\otimes B(A)}(A\otimes B(A)), [d, \; ])\\
 & \stackrel{\cong \vee }{\longrightarrow} &  (End_{\Omega (A)\otimes A^{\vee}}(\Omega (A)\otimes A^{\vee}), [d^{\vee }, \; ])\\
  &  \stackrel{\cong \varphi _{\Omega (A)} }{\longleftarrow} & (Hom_{\Omega(A)}(\Omega (A)\otimes A^{\vee}, \Omega (A)), d). 
\end{array}
$$ 
(Recall that $\varphi _{A}(f)= (f\otimes id_{B(A)})\circ (id_A\otimes \Delta _{B(A)})$ and the same is for $\varphi _{\Omega (A)}$).

We  denote  $B(A)^{\geq q}=  \bigoplus _{k\geq q} T^k(s\bar A)$ and $\Omega (A)^{\geq q} = \bigoplus _{k\geq q}T^k(s^{-1}\overline{A^{\vee}})$.

Consider  on $(Hom_{\Omega(A)}(\Omega (A)\otimes A^{\vee}, \Omega (A)), \mathcal{D})$ the filtration given by:
\begin{eqnarray} \label{11}
\mathcal{I\!\!F'}^q=\{ f \mid f(\Omega (A)\otimes A^{\vee}) \subseteq \Omega (A)^{\geq q} \}.
\end{eqnarray} 
 We have successively for any  $f\in \mathcal{I\!\!F}^q$:
 $$\varphi _{A} (f)(A\otimes B(A))\subseteq A\otimes B(A)^{\geq q} \;\;\;\; and $$
$$(\vee \circ \varphi _{A} ) (f))(\Omega (A)\otimes A^{\vee})=(\varphi _{A} (f))^{\vee}(\Omega (A)\otimes A^{\vee})\subseteq \Omega (A)^{\geq q}\otimes A^{\vee}.$$

Put $h= (\varphi _{A} (f))^{\vee}$. It is clear that 
$$h(\Omega (A)\otimes A^{\vee})\subseteq \Omega (A)^{\geq q}\otimes A^{\vee} \Longleftrightarrow \varphi _{\Omega (A)}^{-1}(h)(\Omega (A)\otimes A^{\vee})\subseteq \Omega (A)^{\geq q}.$$

It results that the composite isomorphism of graded complexes\\
$(Hom_A(A\otimes B(A), A), d)  \stackrel{\cong }{\longrightarrow} (Hom_{\Omega(A)}(\Omega (A)\otimes A^{\vee}, \Omega (A), d)$
 is a filtration preserving.
 
\item  

 Consider the  Adams-Hilton model $(U\mathbb{L}_W, d)$ of $X$ (\cite{FHT01}, Rem. \S 26 (b)), that is $U\mathbb{L}_W$ is quasi-isomorphic with $A$.  It follows an induced quasi-isomorphism $(\Omega (A), d)\stackrel{\simeq}{\longrightarrow}  (\Omega (U\mathbb{L}_W),d)$ preserving word length filtrations. Moreover, the dual of the quasi-isomorphism $C_*(\mathbb{L}_W)\stackrel{\simeq}{\longrightarrow} B(U\mathbb{L}_W)$  (\cite{FHT01}, Prop. 22.7) is a preserving  filtration quasi-isomorphism $ (\Omega (U\mathbb{L}_W), d) \stackrel{\simeq}{\longrightarrow} (C^*(\mathbb{L}_W), d)$. Recall that $(C^*(\mathbb{L}_W), d)$ is a cochain model of $A_{PL}(X)$. Since  $X$ is simply connected and of finite type, $\mathbb{L}_W$ is a connected  Lie algebra of finite type and then $(C^*(\mathbb{L}_W), d)$ is itself a Sullivan algebra. Hence the quasi-isomorphism $(C^*(\mathbb{L}_W), d)  \stackrel{\simeq}{\longrightarrow} A_{PL}(X)$ lefts to a quasi-isomorphism $(C^*(\mathbb{L}_W), d)  \stackrel{\simeq}{\longrightarrow} (\Lambda V,d)$ (where $(\Lambda V,d)$ designate  any minimal model of $X$).   We know that  any morphism of Sullivan algebras is automatically a filtration preserving, so  the last one is. We  conclude that there is a filtration preserving quasi-isomorphism $\varphi : (\Omega (A), d)\stackrel{\simeq}{\longrightarrow} (\Lambda V,d)$.  

\item  
Consider on  $(\Lambda V \otimes \Gamma (sV), d)$  the structure of  $(\Omega (A),d)$-module induced by $\varphi$.  The following diagram:
  $$\begin{array}{ccc}
  & & (\Lambda V \otimes \Gamma (sV), d)\\
  & & \downarrow \simeq \\
  (\Omega (A)\otimes A^{\vee}, d) & \stackrel{\simeq}{\longrightarrow} &  \mathbb{Q}
  \end{array}
  $$
  is completed (cf.  \cite{FHT01}, Prop. 6.4) by a quasi-isomorphism \\ $\Phi : (\Omega (A)\otimes A^{\vee}, d)  \stackrel{\simeq}{\longrightarrow} (\Lambda V \otimes \Gamma (sV), d)$ between $(\Omega (A),d)$-modules  which are semifree resolutions of $\mathbb{Q}$. The filtration of $\Omega (A)\otimes A^{\vee}$ by the  $\Omega (A)$-submodules  $\Omega (A)\otimes A^{\vee}_{\leq i}$ exhibits $\Omega (A)\otimes A^{\vee}$ as an  $\Omega (A)$-semifree  resolution of $\mathbb{Q}$.
% \nocite{McL} 

Consider now on $ Hom_{\Omega (A)}(\Omega (A)\otimes A^{\vee} , \Lambda V), \mathcal{D})$ the filtration 
\begin{eqnarray} \label{12}
\mathcal{I\!\!F''}^q=\{ f \mid f(\Omega (A)\otimes A^{\vee}) \subseteq (\Lambda V)^{\geq q} \}.
\end{eqnarray}
  The dga-morphism  $$
  \Psi: Hom_{\Omega (A)}(\Omega (A)\otimes A^{\vee} , \Omega (A), \mathcal{D})  \rightarrow  Hom_{\Omega (A)}(\Omega (A)\otimes A^{\vee} , \Lambda V), \mathcal{D})
 $$ (resp. 
           $$
  \Psi':  Hom_{\Lambda V}(\Lambda V \otimes \Gamma (sV), \Lambda V), \mathcal{D})  \rightarrow Hom_{\Omega (A)}(\Omega (A)\otimes A^{\vee} , \Lambda V), \mathcal{D})
  $$
 defined by   $\Psi (f) =  \varphi \circ f$ (resp.  $\Psi ' (g) =  g \circ \Phi$) clearly    preserves filtrations \ref{11} and \ref{12} (resp. \ref{9} and \ref{12}). It follows the  induced morphisms of respective spectral sequences.
 
For the remainder we denote $(\Omega (A), d) = (T(W), d)$. By  Proposition 3.6 in \cite{HL88},  $\varphi : (\Omega (A), d)\stackrel{\simeq}{\longrightarrow} (\Lambda V,d)$ induces  the quasi-isomorphism $$E_2(\varphi) : (T(W), d_2)\stackrel{\simeq}{\longrightarrow} (\Lambda V,d_2).$$ Whence (\cite{FHT88}, Remark 1.3 (1)) the second terms of spectral sequences induced by filtrations  \ref{11}, \ref{12} and \ref{9} 
  are   isomorphic.
\end{enumerate}

Combining all this steps, we deduce that   the two spectral sequences \ref{4} and \ref{2} are  isomorphic. 

\end{pf}

\centerline{\bf Proof of theorem $1.0.6$}

Recall  first that the chain map    $ev:  (Hom_{\Lambda V}(\Lambda V \otimes \Gamma (sV), \Lambda V), \mathcal{D}) \longrightarrow (\Lambda V,d)$ which induces  $ev_{(\Lambda V,d)}$  is compatible with filtration  \ref{8} and \ref{9}. In fact  it is a morphism of filtered cochain complexes.
  
Before beginning the proof, we outline in the following remark some proprieties of the invariant $\textsc{r}(\Lambda V,d)$ which are similar to those of $e_{0}(\Lambda V,d)$.
\begin{rem} The filtration inducing the spectral sequence \ref{4} is \\ $\mathcal{F}^p=\{f\in Hom_{\Lambda V}(\Lambda V\otimes \Gamma(sV),\Lambda V)\; \mid \; f(\Gamma (sV))\subseteq \Lambda ^{\geq p}V \}.$
We deduce immediately the following:
\begin{enumerate}
\item If $(\Lambda V, d)$ is a Sullivan minimal model, then  $\textsc{r}(\Lambda V, d)$  is the largest integer $p$ such that some nontrivial  class in  $\mathcal{E}xt_{(\Lambda V, d)}^{*}(\mathbb Q, (\Lambda V, d))$ is represented by a cocycle in $\mathcal{F}^{ p}$. Equivalently it is the least integer $p$ such that the projection $ \mathcal{A} \rightarrow \mathcal{A}/\mathcal{F}^{> p}$ induces an injection in cohomology, where $\mathcal{A} = Hom_{(\Lambda V,d)}[(\Lambda V\otimes \Gamma (sV), D), (\Lambda V,d)]$.
\item Suppose $dim(V)<\infty $, so that $(\Lambda V,d)$ is a Gorenstein algebra (\cite{LM01}), ie $\mathcal{E}xt_{(\Lambda V, d)}^{*}(\mathbb Q, (\Lambda V, d))$ is one dimensional (\cite{Mur94}). Denote by $\Omega $ its generator. The projection  $ \pi : \mathcal{A} \rightarrow \mathcal{A}/\mathcal{F}^{> p}$  is then an injection in cohomology if and only if $H(\pi ) (\Omega )\not = 0$. Therefore $$\textsc{r}(\Lambda V,d)= sup\{ p \mid \Omega \; \hbox{can be represented by a cocycle in}\;  \mathcal{F}^{ p}\}.$$
\end{enumerate}
\end{rem}
\begin{pf} (of Theorem 1.0.6.):

We denote, as in the  introduction, $(\Lambda V,d)$ the Sullivan minimal model of $X$. Since $(\Lambda V,d)$ is elliptic, $e_0(\Lambda V,d)= cat_{\mathbb{Q}} (X)$ (\cite{FHL}), so it suffices to prove that $e_0(\Lambda V,d) \geq \textsc{r}(\Lambda V,d) + (k-2)$. Since $dim(V)<\infty $, (\cite{LM01}) 
   $(\Lambda V,d_k)$ and $(\Lambda V,d)$ are   Gorenstein graded algebras.  So  there exists a unique $(p,q)\in \mathbb{N}\times \mathbb{N}$  such that $ \mathcal{E}xt^*_{(\Lambda V,d_k)}(\mathbb{Q} , (\Lambda V,d_k))=\mathcal {E}xt^{p,q}_{(\Lambda V,d_k)}(\mathbb{Q} , (\Lambda V,d_k))$, with a unique generator. The $\mathcal{E}_{\infty} $ term of   \ref{4} is then a one-dimensional $\mathbb{Q}$-vector space concentrated in the bidegree $(p,q)$. It results that  $\textsc{r}(\Lambda V,d)=p$.

Also, the convergence of \ref{4} implies that $ \mathcal {E}_{\infty}^{p,q} \cong  \mathcal {E}xt^{p+q}_{(\Lambda V,d)}(\mathbb{Q} , (\Lambda V,d))$, hence $\mathcal {E}xt^{*}_{(\Lambda V,d)}(\mathbb{Q} , (\Lambda V,d))$ is concentrated in degree $p+q$ and then  the formal dimension of $(\Lambda V, d)$ is exactly $N = p+q$ (in fact,  by ellipticity of $(\Lambda V,d)$, we have $0\not = ev_{(\Lambda V,d)}  : \mathcal {E}xt^{p+q}_{(\Lambda V,d)}(\mathbb{Q} , (\Lambda V,d))  \stackrel{\cong}{\rightarrow}   H^{p+q}(\Lambda V,d).$ Therefore  $N = p+q $ equals to degree of the fundamental class of $(\Lambda V,d )$).

 To finish this proof, we need the following lemma:
\begin{lem}
 If $[h]$ is  any generator of $\mathcal {E}xt^{p,q}_{(\Lambda V,d_k)}(\mathbb{Q} , (\Lambda V,d_k))$, then $ h(1)\in \Lambda ^{\geq  p'}V$, with  $p' = p + (k-2)\geq p$.
\end{lem}   
\begin{pf}
 $[h]$ being a generator of $\mathcal {E}xt^{p,q}_{(\Lambda V,d_k)}(\mathbb{Q} , (\Lambda V,d_k))$ implies that $h(\Gamma (sV))\subseteq \Lambda ^{\geq p}V$. Now   let $x_i$ an generator  of $V$ with  the smallest  degree.  Necessarily $d_k(x_i) = 0$, hence $D_k(sx_i) = x_i$ which leads  $ d_k(h(sx_i))= \pm  h(D_k(sx_i)) = h(x_i) = x_ih(1)$. Since $d_k(h(sx_i))\in \Lambda ^{\geq p + (k-1)}V$, we have  $h(1)\in \Lambda ^{\geq p + (k-2)}V$. Finally $h(1)\in \Lambda ^{\geq  p'}V$, with $p' = p +(k-2) \geq p$.
\end{pf} 
As a consequence, there exists $(p',q') \in \mathbb{N}^2$ such that $p'\geq p$, $p' +q' = p + q = N$ and  $ev_ {(\Lambda V,d_k)} : \mathcal {E}xt^{p,q}_{(\Lambda V,d_k)}(\mathbb{Q} , (\Lambda V,d_k)) {\longrightarrow}  H^{p', q'}(\Lambda V,d_k)$. This map induces $ev_{\infty} : \mathcal {E}^{p,q}_{\infty} \rightarrow E_{\infty}^{p',q'} .$

Once again,  $(\Lambda V,d)\;  \hbox{is elliptic}\;  \Leftrightarrow ev_{(\Lambda V,d)}\not = 0$ (cf. \cite{Mur94}), thus using compatibility of filtrations with $ev_{(\Lambda V,d)}$ and the  convergence of spectral sequences, we deduce that  $ev_{\infty }\not = 0$ and then $E_{\infty}^{p',q'}\not = 0$.  This  is illustrated by the following  commutative diagram:
 $$\begin{array}{ccc}
   \mathcal {E}^{p,q}_{\infty} &  \stackrel{\cong}{\rightarrow} &  \mathcal {E}xt^{p+q}_{(\Lambda V,d)}(\mathbb{Q} , (\Lambda V,d))\\
   ev_{\infty } \downarrow &  &  ev_{(\Lambda V,d)}\downarrow \\
     E_{\infty}^{p',q'} & \stackrel{\cong}{\rightarrow} & H^{p+q}(\Lambda V,d)
  \end{array}
$$

 To conclude,  by remark $2.3.1. (2)$, we see  that  the Milnor-Moore spectral sequence and its generalization \ref{3} define both the Tommer invariant $e_0(\Lambda V,d) $.  Hence $e_0(\Lambda V,d) \geq  p' = \textsc{r}(\Lambda V,d) + (k-2)$.

 \end{pf}

 \centerline{\bf Proof of theorem $1.0.8$}

 \begin{pf}
As it is noted in the introduction, it suffice to calculate the invariant $\textsc{r}(\Lambda V,d)$ for a minimal model whose $dim(V)<\infty$ and the differential being pure and homogeneous of a certain degree $k$. 

Since $dim(V)< \infty$,  $(\Lambda V,d)$ is a Gorenstein algebra, that is \\ $dim_{\mathbb{Q}}(\mathcal {E}xt^{*}_{(\Lambda V,d)}(\mathbb{Q} , (\Lambda V,d)))= 1$. Also,  the differential being  pure,  there exists (\cite{Mur93}) 
$h : (\Lambda V\otimes \Gamma sV , D)\rightarrow (\Lambda V, d)$,  a $(\Lambda V,d)$-morphism s.t. $[h(1)]=ev_{(\Lambda V,d)} (h)$ is the top cohomology class of $(\Lambda V,d)$ (\cite{LM01}).

Denote by $\omega$ the cycle   representing this top cohomology class and let $x_1,  \ldots , x_n$ and $y_1, \ldots , y_m$  the generators of $V^{enen}$ and $V^{odd}$ respectively.
First, note that  $\omega $ has word length equal to $p = m+n(k-2)$ (\cite{LM01}). Second,  $\mathcal {E}xt^{*}_{(\Lambda V,d)}(\mathbb{Q} , (\Lambda V,d)) = Im(ev_{(\Lambda V,d)})$   (because  they are both  one dimensional). So any element  \\ $[\varphi] \in Im(ev_{(\Lambda V,d)})$ is such that
$\varphi (1)=\omega + d(\omega ')$.  

In the sequel, we will determine $\varphi (\Gamma (sV)$ for an arbitrary an arbitrary $\varphi \in \Gamma (sV)$. Remark first that $\varphi (x_i)=x_i\varphi (1) $ and $\varphi (y_j)=y_j\varphi (1)$;
$1\leq i \leq n$ and $1\leq j \leq m$.

To finish, we will discus  tow cases:
\begin{enumerate}[Step 1.]
\item  Assume that $ d(\omega ') \in \Lambda ^{\geq p}V$.
As $d$ is pure, we have $D(sx_i) = x_i$ and $D(sy_j) = y_j + s(dy_j)$ hence 
$d(\varphi (sx_i)) = \varphi (D(sx_i)) = \varphi (x_i) = x_i(\omega + d(\omega ')) \in \Lambda ^{\geq p+1}V$.    This implies that $\varphi (sx_i)\in \Lambda ^{\geq (p+1)-(k-1)}V=\Lambda ^{\geq p-k+2}V $. 
Also, $d(\varphi (sy_j)) = \varphi (D(sy_j)) = \varphi (y_j+s(dy_j)) = y_j(\omega + d(\omega '))+\varphi (s(dy_j))$. 
 Moreover $$dy_j = \sum _{j_1, \ldots ,j_k}x_{j_1}x_{j_2}\ldots x_{j_k} \in \Lambda ^{k}V \;\;\;\;  \hbox{with} {j_1} < {j_2} < \ldots <  {j_k}$$
 so that  $$s(dy_j))=\frac{1}{k}\sum _{j_1, \ldots ,j_k}\sum _{l=1}^{l=k}x_{j_1}x_{j_2}\ldots x_{j_{l-1}}(sx_{j_l})x_{j_{l+1}}\ldots x_{j_k}.$$ 
Therefore, $$\varphi (s(dy_j))= \frac{1}{k}\sum _{j_1, \ldots ,j_k}\sum _{l=1}^{l=k}x_{j_1}x_{j_2}\ldots x_{j_{l-1}}x_{j_{l+1}}\ldots x_{j_k}\varphi (sx_{j_l}).$$
We deduce  that $d(\varphi (sy_j))\in \Lambda ^{\geq p+1}V$ from which $\varphi (sy_j)\in \Lambda^{\geq p-k+2}V$. Thereafter, using the algebraic structure of $\Lambda V \otimes \Gamma (sV)$ and minimality, we deduce by induction that any product in $ \Gamma (sV)$ is sent by $\varphi $ in $\Lambda ^{\geq p-k+2} V$. 
Consequently, for such $\varphi $, $\varphi (\Gamma (sV))\subseteq \Lambda ^{\geq p-k+2}V$.
\item  Suppose that $ d(\omega ') \in \Lambda ^{\geq q}V$ with $ q < p $, the same argument as in the first step gives   
$\varphi (\Gamma (sV))\subseteq \Lambda ^{\geq q-k+2}$. 
\end{enumerate}

We conclude then that: $\textsc{r}(\Lambda V,d)=p-k+2 =m+(n-1)(k-2)$.
\end{pf}

\section{Final remark}
\begin{enumerate}
%\begin{rem}
\item  Denote  $p = \textsc{r}(\Lambda V, d)$. By Remark 2.3.1. (2) and the proof of theorem 1.0.6. we state the following:    $e_0(\Lambda V, d) = \textsc{r}(\Lambda V, d) + (k-2)$  if and only if there exists a  generator $[h]$ of $\mathcal {E}_{\infty}^{p,q}  = \mathcal{E}xt_{(\Lambda V, d_k)}^{p,q}(\mathbb Q, (\Lambda V,  d_k))$  whose  image $[h']$ by the isomorphism $ \mathcal {E}_{\infty}^{p,q} \stackrel{\cong}{\rightarrow}   \mathcal {E}xt^{p+q}_{(\Lambda V,d)}(\mathbb{Q} , (\Lambda V,d))$  gives a  cocycle $h'(1)$ that realises $e_0(\Lambda V, d)$.

  Regarding to Proposition 3. in \cite{LM02} and the previous  theorem, we see that if $d$ is  pure and non homogeneous,  the equality  can't hold unless $(\Lambda V, d_k)$ is elliptic.

\item With the same notations as in  the last proof, we remark that for any $[\varphi]$ in $\mathcal {E}xt^{*}_{(\Lambda V,d)}(\mathbb{Q} , (\Lambda V,d))$ (the differential being pure and homogeneous of degree $k$) $\varphi (1)\in \Lambda ^{p}V$ while $\varphi (\Gamma ^+(sV))\subseteq \Lambda ^{\geq q -k + 2}V$. Now for a Sullivan minimal algebra $(\Lambda V, d)$, with $d: V\rightarrow \Lambda ^{\geq k}V$ and  $dim(V)<\infty $, using the convergence of \ref{7} for the induced  model $(\Lambda V, d_k)$ and remark $2.1.1$, we obtain a generator $\varrho $ for $\mathcal {E}xt^{*}_{(\Lambda V,d_k)}(\mathbb{Q} , (\Lambda V,d_k))$ such that $\varrho (1)\in \Lambda ^{\geq p}V$. Also, using the convergence of \ref{4}, we obtain a  generator $[\vartheta ]$ of   $\mathcal {E}xt^{*}_{(\Lambda V,d)}(\mathbb{Q} , (\Lambda V,d))$, such that $\vartheta (1)\in \Lambda ^{\geq p}V$. As a consequence, if in addition, we suppose  $(\Lambda V, d)$ elliptic,  we deduce that $e_0(\Lambda V,d)\geq p = m+n(k-2)$, so we  recover the result of S. Ghorbal and B. Jessup (\cite{GJ01}, cor.3 and \cite{Lup02}, Rem 2.4). 

Furthermore, the determination of $\varrho (\Gamma ^+(sV))$ and $\vartheta (\Gamma ^+(sV))$  depends on the calculation of the images $\varrho (D(sx_i))$ and $\varrho (D(sy_j))$ (resp $\vartheta (D(sx_i))$ and $\vartheta (D(sy_j)$) which seems more difficult. Obviously, in doing such  calculations, one must use spectral sequences \ref{3} and \ref{6}. The following question is natural.

{\it Question}: For a  Sullivan minimal  algebra $(\Lambda V, d)$, with $d: V\rightarrow \Lambda ^{\geq k}V$ and  $dim(V)<\infty $, is $\textsc{r}(\Lambda V,d)\geq  m+(n-1)(k-2)?$  

\item 
Let $\mathbb K$  a field of $char(\mathbb K)= p >2$ and  $X$ an $r$-connected $\mathbb K$-elliptic finite CW-complex in the rang of Anick. Let
$(\Lambda V, d)$ denote its Sullivan minimal model. By the same argument as in the rational case, we have
 $$cat(X) \geq e_{\mathbb K}(X)\geq \textsc{r}(\Lambda V,d) + (k-2).$$
 \end{enumerate}
 %\end{document}
 \section{Acknowledgements} I am indebted to J. C. Thomas  for  very useful conversations which enabled me to improve my results significantly. I am also grateful to the reviewers who allowed me with their comments to improve the drafting of this work.
 
\end{document}